\documentclass[12pt]{article}

\usepackage{amsthm,amsmath,amssymb}

\usepackage{graphicx}

\usepackage[colorlinks=true,citecolor=black,linkcolor=black,urlcolor=blue]{hyperref}

\usepackage{enumitem}
\setlist{nolistsep}

\newcommand{\arxiv}[1]{\href{http://arxiv.org/abs/#1}{\texttt{arXiv:#1}}}
\newcommand{\conv}{\textrm{Conv}}
\newcommand{\euler}[2]{\genfrac{\langle}{\rangle}{0pt}{}{#1}{#2}}

\newcommand{\Ver}{\textrm{Vert}}
\newcommand{\link}{\mathrm{link}}

\theoremstyle{plain}
\newtheorem{theorem}{Theorem}
\newtheorem{lemma}[theorem]{Lemma}
\newtheorem{corollary}[theorem]{Corollary}

\theoremstyle{definition}
\newtheorem{definition}[theorem]{Definition}

\theoremstyle{remark}

\title{\bf On interior polytope number sequences}

\author{Michael A. Jackson\\
\small Department of Mathematics\\[-0.8ex]
\small Grove City College\\[-0.8ex]
\small Grove City, PA, USA\\
\small\tt majackson@gcc.edu
}

\begin{document}

\maketitle


\begin{abstract}
  Polytope numbers for a given polytope are an integer sequence defined by the combinatorics of the polytope. Recent work by H. K. Kim and J. Y. Lee has focused on writing polytope number sequences as sums of simplex number sequences. In addition, these works have given a process for writing the polytope number sequence in a recursive fashion by using the interior sequence for the various $k$-faces of the polytope, each viewed as a $k$-dimensional polytope. This paper shows that the coefficients of the linear combination of simplex number are the $h$-vector components for a certain type of triangulation of the polytope. In addition, reversing the order of the coefficients in the linear combination is shown to equal the interior polytope sequence for this polytope.

\end{abstract}

\section{Introduction}

Polygonal numbers are sequences of integers based on the set of points forming polygons on the plane. These number sequences date back to the ancient Greeks. The polygonal numbers have been generalized to polyhedral numbers by the ancient Greeks and, more recently, to all higher dimensions by Kim \cite{hkk:pn} and then by Kim and Lee \cite{hkk:pnp}. For more on polytope numbers see also \cite{dd:fn}.

To construct the polytope number sequences, suppose that we have a uniform $d$-dimensional polytope ${P}^d$. We construct the sequence $\{P^d(n)\}_{n\geq 1}$ using induction on $n$. By convention for any polytope, $P^d(1)=1$. Suppose that we have constructed $P^d(n-1)$, which is represented by a set of points forming the polytope $X$ in $d$-dimensional Euclidean space where $X$ is similar to ${P}^d$. We take a vertex $\mathbf{x}$ of $X$ and extend the edges containing $\mathbf{x}$ to include one additional point and use these to create a new polytope $Y$, which contains $X$ and is similar to $X$ and ${P}^d$. Then $P^d(n)-P^d(n-1)$ is the additional points added to $X$ to form $Y$. This difference is computed by summing the points in the interior of each  $k$-facet added to $X$ to make $Y$ for each $0\leq k \leq d-1$. Each $k$-facet is treated as a $k$-dimensional polytope and the points in the interior for that polytope are the $n^{\textrm{th}}$ number in that polytope sequence minus all of the points on its exterior. For more on this process see \cite[Section 1]{hkk:pn} and \cite[Section 3]{hkk:pnp}.

H. K. Kim gives formulas for the $d$-dimensional regular polytope number sequences \cite{hkk:pn}. We use $\alpha^d(n)$, $\beta^d (n)$, and $\gamma^d(n)$ to denote the $n^{\textrm{th}}$ number in the polytope sequences for the $d$ dimensional simplex, cross-polytope, and measure-polytope, respectively. With this notation, Kim's formulas are
$$\alpha^d(n)=\binom{n+d-1}{d} \textrm{ , }\hspace{.2 in} \beta ^d(n)=\sum _{i=0}^{d-1} \binom{d-1}{i} \alpha^d(n-i)$$
$$\textrm{ and } \gamma^d(n)=\sum _{i=0}^{d-1} \euler{d}{i} \alpha^d(n-i)$$
where $\euler{d}{i}$ are the Eulerian numbers \cite{hkk:pn}. 
Kim and Lee claim that any polytope number sequence for a $d$-dimensional polytope can be written as $$P^d(n)=\sum _{i=0}^{d-1} a_i \;\alpha^d(n-i)$$
where $a_0=1$ and $a_i\geq 0$ for each $i$\cite[Theorem 4.1]{hkk:pnp}. This formula will be called the $d$-dimensional simplex decomposition of the polytope number sequence. 
\begin{theorem}
\label{thm:premain}
Let $P$ be a $d$-dimensional convex polytope. There exists a triangulation of $P$ such that the polytope number sequence $P^d(n)$ can be written as a linear combination
$$P^d(n)= \sum _{j=0}^{d-1} h_j \; \alpha ^d(n-j)$$
where the $h_j$'s are the components of the $h$-vector of a pointed triangulation of the polytope $P$.
\end{theorem}
If $P^d(n)$ is the polytope number sequence for a given polytope, we use $P^d(n)^\#$ to represent the interior number sequence. This sequence is given by taking $P^d(n)$ and subtracting the points that are on the exterior of this polytope in the $d$-dimensional space it inhabits.

\begin{theorem}
  \label{Thm:Main}
Let $P$ be a $d$-dimensional convex polytope. The interior polytope number sequence $P^d(n)^\#$ can be written as a linear combination
$$P^d(n)^\#= \sum _{j=0}^{d-1} h_j \; \alpha ^d(n-d-1+j)$$
where the $h_j$'s are the components of the $h$-vector of a pointed triangulation of the polytope $P$.
\end{theorem}

This paper is organized as follows: Section 2 gives the preliminaries, including the proof that every convex polytope has a pointed triangulation; Section 3 gives a formal definition for the polytope number sequence and a proof of Theorem \ref{thm:premain}; and the proof of Theorem \ref{Thm:Main} is in Section 4.

\section{Preliminaries}

For a subset $K$ of $\mathbb{R}^d$, let $\conv (K)$ be the convex hull of $K$. A polytope is $\conv (K)$ for some finite subset $K$ of $\mathbb{R}^d$. The dimension of a polytope $P$ is the dimension of its affine hull and is denoted $\dim (P)$.

\begin{definition}
A {\em polytopal complex} is a finite collection $\mathcal{C}$ of polytopes that satisfies the following conditions:
\begin{enumerate}
\item The empty polytope is in $\mathcal{C}$.
\item If $P\in \mathcal{C}$, then every face of $P$ is also in $\mathcal{C}$.
\item The intersection of two polytopes $P$ and $Q$ both in $\mathcal{C}$ is a face of both of the polytopes $P$ and $Q$.
\end{enumerate}
The dimension of $\mathcal{C}$ is defined by $\dim(\mathcal{C})=\max \{\dim (P) | P\in \mathcal{C} \}$ and the underlying set of $\mathcal{C}$ is the set $|\mathcal{C} | = \bigcup _{P\in \mathcal{C}} P$. 
\end{definition}
A polytopal complex $\mathcal{C}$ is called {\em pure} if for each $P\in \mathcal{C}$, there is a $Q\in \mathcal{C}$ with $P\subseteq Q$ and $\dim(Q)=\dim(\mathcal{C})$. A polytopal complex $\mathcal{C}$ is called a simplicial complex if every polytope in $\mathcal{C}$ is a simplex.

A {\em subdivision} of polytope $P$ is a polytopal complex $\mathcal{C}_P$ with the underlying space $|\mathcal{C}_P|=P$. A {\em triangulation} of a polytope $P$ is a subdivision $\mathcal{C}_P$, which is a simplicial complex.

Given a polytopal complex $\mathcal{C}$ of dimension $d$, we can define the set of $k$-faces to be $\mathcal{F}_k(\mathcal{C})=\{ P\in \mathcal{C} | \dim(P)=k \}$. In a $d$-dimensional polytope or polytopal complex, the facets are the $(d-1)$-dimensional faces. We also will denote by $\Ver (P)$ the set of vertices of the polytope $P$.

We will focus on a certain type of triangulation of a polytope called a pointed triangulation; thus we concentrate on simplices. We will use $\alpha ^d$ to represent a simplex of dimension $d$.

\begin{definition}
A triangulation $\mathcal{C}_P$ of a $d$-polytope $P$ will be called a {\em pointed triangulation} if
\begin{enumerate}
\item For each $k\in \{0,1,\dots d\}$, each $k$-face $F$ of $P$ has a triangulation $\mathcal{C}_F$ such that there is a designated vertex $\mathbf{v}_F\in \Ver(F)$, called the {\em apex} of $F$, satisfying $\mathbf{v}_F\in \alpha^k_i$ for all $\alpha^k_i\in \mathcal{C}_F$. (i.e. the apex of $\mathcal{C}_F$ is contained in every $k$-simplex in the triangulation $\mathcal{C}_F$.)
\item For any two faces $F_1$ and $F_2$ of $P$, if $\{\mathbf{v}_{F_1},\mathbf{v}_{F_2}\}\subset F_1\cap F_2$, then $\mathbf{v}_{F_1}=\mathbf{v}_{F_2}$.
\item For each face $F$ of $P$, if $\mathbf{w}\in \Ver(F)\setminus \{\mathbf{v}_{F}\}$, then the edge $\conv(\{\mathbf{v}_{F}, \mathbf{w}\})\in \mathcal{C}_F$.
\end{enumerate}
\end{definition}
We also define $V(P)=\{\mathbf{v}_{F} | F \textrm{ is a face of }P\}$. $V(P)$ is called the set of apexes for the pointed triangulation $\mathcal{C}_P$ and depends on the pointed triangulation. We call $\mathcal{C}_P$ the $V(P)$-pointed triangulation.

Kim and Lee have shown that every convex polytope has a pointed triangulation \cite[Theorem 2.1]{hkk:pnp}. We will give the proof again here for completeness.

\begin{theorem}[Kim and Lee, {\cite[Theorem 2.1]{hkk:pnp}}]
\label{thm:ptri}
Every polytope has a pointed triangulation.
\end{theorem}
\begin{proof}
Let $P\subset \mathbb{R}^n$ be a $d$-polytope. There exists a linear function in general position $\mathbf{cx}$ with respect to $P$ (i.e. a linear function such that for all $\mathbf{v}\neq \mathbf{w}\in \Ver (P)$, $\mathbf{cv}\neq \mathbf{cw}$). See \cite[Lemma 1.5]{hkk:pnp} and \cite[Lemma 3.4]{gz:lop}. For each face $F\in \mathcal{F}(P)$, define $\mathbf{v}_F\in \Ver (F)$ to be the vertex of $F$ with the minimum value of the linear function (i.e. $\mathbf{c} \mathbf{v}_F < \mathbf{c w}$ for all $\mathbf{w}\in \Ver (F)\setminus \{\mathbf{v}_F \}$).

Now for each $F\in \mathcal{F}(P)$, we define $V(F)=\{ \mathbf{v}_G \vert G\in \mathcal{F}(F) \}$. We also define 
$$\mathcal{C}_F = \{\emptyset \} \cup \left\{ \left. \conv ( \{\mathbf{v}_{G_1}, \mathbf{v}_{G_2}, \dots , \mathbf{v}_{G_k} \} ) \right| k\in [d], G_i \in \mathcal{F}(F), G_i \supsetneq G_{i+1}, \mathbf{v}_{G_i} \notin G_{i+1} \right\}.$$

We will now show that $\mathcal{C}_P$ is the $V(P)$-pointed triangulation. First we will show that $\mathcal{C}_P$ is a triangulation of $P$. By definition, $\mathcal{C}_P$ is a simplicial complex and $|\mathcal{C}_P|\subseteq P$, so we are left to show that $|\mathcal{C}_P|\supseteq P$. We will proceed by induction on the dimension. By the definition, we see that $$P=\bigcap _{\substack{F\in \mathcal{F}_{d-1}(P)\\\mathbf{v}_P\in F}} \conv ( \{\mathbf{v}_P \} \cup F ).$$
Suppose that $\mathbf{x}\in P$; then there exists $F\in \mathcal{F}_{d-1}(P)$ such that $\mathbf{v}_P\notin F$ and $\mathbf{x}\in \conv ( \{\mathbf{v}_P \} \cup F )$. By the induction hypothesis $|\mathcal{C}_F | = F$ and so there exists a $(d-1)$-simplex $\alpha ^{d-1}$ in $\mathcal{C}_F$ with $\mathbf{x}\in \conv (\{\mathbf{v}_P\} \cup \alpha ^{d-1} )$. By the definition of $\mathcal{C}_P$, $\conv (\{\mathbf{v}_P\} \cup \alpha ^{d-1} )$ is an element of $\mathcal{C}_P$ and thus $x\in |\mathcal{C}_P |$.

We are left now to show that the triangulation $\mathcal{C}_P$ is the $V(P)$-pointed triangulation. By the definition of $\mathcal{C}_P$, conditions 1 and 3 of the pointed triangle definition are clearly satisfied. To see condition 2, let $F_1$ and $F_2$ be two faces of $P$ such that $\{\mathbf{v}_{F_1}, \mathbf{v}_{F_2} \} \subseteq F_1 \cap F_2$. Since $F_1\cap F_2$ is a face of $P$ that is contained in both $F_1$ and $F_2$, we see that $\mathbf{v}_{F_1\cap F_2}$ must be the same as $\mathbf{v}_{F_1}$ and $\mathbf{v}_{F_2}$. This shows that $\mathcal{C}_P$ is the $V(P)$-pointed triangulation.
\end{proof}
Throughout the rest of this paper, we will use the $V(P)$-pointed triangulation $\mathcal{C}_P$ for each convex polytope $P$.

\begin{definition}
For a convex polytope $P$, let $\mathcal{C}_P$ be the $V(P)$-pointed triangulation. We will define the exterior polytopal complex $\mathcal{C}_{\partial P}$ to be the collection of faces $F\in \mathcal{C}_P$ such that $F\in G$ for some $G\in \mathcal{F} (P)\setminus \{P\}$.
\end{definition} 

\begin{definition}[See page 237 of \cite{gz:lop}]
Given a polytopal complex $\mathcal{C}$ and a vertex $\mathbf{v}$ of $P$, we define the star of the vertex $\mathbf{v}$, denoted $\textrm{star}(\mathbf{v},\mathcal{C})$, to be the polytopal subcomplex of all faces that contain $\mathbf{v}$ and their faces. We also define the link of $\mathbf{v}$, denoted $\link(\mathbf{v},\mathcal{C})$, to be the subcomplex of all faces in $\mathrm{star}(\mathbf{v},\mathcal{C})$ that do not have $\mathbf{v}$ as a vertex. 
\end{definition}
If $\mathcal{C}$ is pure of dimension $d$, then so is $\mathrm{star}(\mathbf{v},\mathcal{C})$, and $\link (\mathbf{v},\mathcal{C})$  is pure of dimension $d-1$.
\begin{definition}
Let $\mathcal{C}$ be a pure $d$-dimensional simplicial complex. A partition of $\mathcal{C}$ is a disjoint collection of intervals $[G_F,F]$ for each $F\in \mathcal{F}_d(\mathcal{C})$ ($[G_F,F]=\{ G : G_F \subseteq G \subseteq F \}$) where $\mathcal{F}=\cap _{F\in \mathcal{F}_d(\mathcal{C})} [G_F,F]$; if $F\neq F' \in \mathcal{F}_d(\mathcal{C})$, then $[G_F,F]\cap [G_{F'},F']=\emptyset$.
A pure simplicial complex that has a partion is called partitionable.
\end{definition}

\begin{theorem}[See \cite{ks:nrssp} and Proposition 2.8 of \cite{rs:cca}]
\label{thm:part}
Let $\mathcal{C}$ be a pure simplicial complex such that the geometric realization $|\mathcal{C}|$ is convex; then $\mathcal{C}$ is partitionable.
\end{theorem}
We will include the proof for completeness, but it is essentially the same as \cite{ks:nrssp} and as outlined in \cite{rs:cca}.
\begin{proof}
Let $P$ be the geometric realization of $\mathcal{C}$ and let $\mathbf{x}$ be a generic point in $P$. (We need $\mathbf{x}$ to not be contained in the affine hull of any simplex in $\mathcal{C}$.) For each $F\in \mathcal{F}_d \mathcal{C}$, let $R_F$ be the set of facets of $F$ visible from $\mathbf{x}$. (By ``visible" we mean that a ray from $\mathcal{x}$ to the interior of the facet will not intersect $F$ before reaching the facet.) Now let $G_F$ be the set of vertices of $F$ that are each opposite one facet in $R_F$. Notice that if $F$ is the $d$-simplex containing $\mathbf{x}$, then $R_F=G_F=\emptyset $.

We claim that the intervals $[G_F,F]$ for each $F\in \mathcal{F}_d (\mathcal{C})$ form a partition of $\mathcal{C}$. Let $H$ be a face of $\mathcal{C}$. Clearly there is a unique $d$-simplex $F_0$ so that the line segment $[\mathbf{x},\mathbf{y}]$ intersects $F_0$ at points before $\mathbf{y}$ for all $\mathbf{y}$ in the interior of $H$ (i.e. not in any face of smaller dimension). Notice that any facet of $F_0$ either contains $H$ or is opposite a vertex contained in $H$. The facets of $F_0$ that contain $H$ are exactly those that are not visible from $\mathbf{x}$. This shows that $H  \in [G_{F_0}, F_0 ]$.

Suppose that $H\in [G_F,F]$ for some $F\in \mathcal{F}_d (\mathcal{C})$. Then $G_F\subseteq H$, which implies that the $H$ is not contained in any facet of $F$ visible from $\mathbf{x}$. Thus if we take an interior point $\mathbf{y}$ of $H$ and move into any facet of $F$ containing $H$, the ray from $\mathbf{x}$ will encounter $F$ before reaching that facet. This implies that the line segment $[\mathbf{x},\mathbf{y}]$ intersects $F$ before the point $\mathbf{y}$. Then $F$ is the $F_0$ from the previous paragraph. In this way if $F, F'\in \mathcal{F}_d (\mathcal{C})$ are not the same, then $[G_{F},F]\cap [G_{F'},F'] = \emptyset $. This finishes the proof that we have a partition of the complex $\mathcal{C}$. 
\end{proof}
\begin{corollary}
\label{cor:part}
Let $P$ be a convex polytope with a $V(P)$-pointed triangulation $\mathcal{C}_P$. The pointed triangulation $\mathcal{C}_P$ is partitionable.
\end{corollary}

\begin{definition}\label{def:fvec}
The $f$-vector of a $d$-dimensional polyhedral complex $\mathcal{C}$ is the vector 
$$\mathbf{f}(\mathcal{C})= ( f_{-1}, f_0, f_1, \dots, f_d ) \in \mathbb{N}^{d+2},$$
where $f_k=f_k(\mathcal{C})$ denotes the number of $k$-dimensional faces in $\mathcal{C}$.
\end{definition}

\begin{definition}\label{def:hvec}
The $h$-vector of a $d$-dimensional simplicial complex $\mathcal{C}$ is the vector
$$\mathbf{h}(\mathcal{C})=( h_0, h_1, \dots , h_d, h_{d+1} )\in \mathbb{Z}^{d+2},$$ 
given by the formula 
$$h_k:= \sum _{i=0}^k (-1)^{k-i} \binom{d-i}{d-k} f_{i-1}.$$
\end{definition}

There is another well known equivalent defintion of the $h$-vector of a partitionable simplicial complex as stated in Section 8.3 of Ziegler \cite{gz:lop}.
\begin{theorem}[{Section 8.2 of \cite{gz:lop}}]\label{thm:hvec}
Let $\mathcal{C}$ be a pure $d$-dimensional simplicial complex and let the collection $[G_F,F]$ for each $F\in \mathcal{F}_d (\mathcal{C})$ be a partition of $\mathcal{C}$. Then for each $i\in [d+1]$, 
$$ h_i (\mathcal{C})= \left| \left\{ F : |G_F|=i , F\in \mathcal{F}_d (\mathcal{C}) \right\} \right|.$$
\end{theorem}

\begin{lemma}
\label{lem:eq}
Let $P$ be a $d$-polytope with a $V(P)$-pointed triangulation $\mathcal{C}_P$. Then $\link(\mathbf{v_P},\mathcal{C}_P)$ is partitionable. In addition, $h_{d+1}(\mathcal{C}_P)=0$ and for $i\in [d]$, $h_i(\mathcal{C})=h_i(\link(\mathbf{v_P},\mathcal{C}_P))$. 
\end{lemma}
\begin{proof}
We will use $\alpha _1^d, \alpha _2^d, \dots \alpha _s ^d$ for the $d$-simplices of $\mathcal{C}_P$. Since $\mathcal{C}_P$ is a $V(P)$-pointed triangulation, for each $i$ there is an $\alpha _i ^{d-1}\in \link (\mathbf{v_P},\mathcal{C}_P)$ with $\alpha _i ^d = \conv (\{\mathbf{v}_P \} \cup \alpha _ i ^{d-1})$. 

Suppose that the collection $[G_{\alpha _i^d} , \alpha _i^d]$ is a partition of $\mathcal{C}_P$. Since $\alpha _i ^{d-1}$ is in only one $d$-simplex of $\mathcal{C}_P$, namely $\alpha _i ^{d}$, it is clear that $G_{\alpha _i^d}\subset \alpha _i ^{d-1}$ for each $i$. We claim that the collection of intervals $[G_{\alpha _i^d} , \alpha _i^{d-1}]$ is a partition of $\link(\mathbf{v_P},\mathcal{C}_P)$. Disjointness of the intervals follows from the disjointness of the partition of $\mathcal{C}_P$ since $[G_{\alpha _i^d} , \alpha _i^{d-1}]\subseteq [G_{\alpha _i^d} , \alpha _i^d]$ for each $i$. To see that the union covers the link, we consider $H\in \link(\mathbf{v_P},\mathcal{C}_P)$, but then $H\in \mathcal{C}_P$ with $\mathbf{v}_P\notin H$. Because $H\in \mathcal{C}_P$, there is an $i$ with $H\in [G_{\alpha _i^d} , \alpha _i^d]$; however, since $\mathbf{v}_P\notin H$, $H\in \alpha _i ^{d-1}$ and thus $H\in [G_{\alpha _i^d} , \alpha _i^{d-1}]$.

By Theorem \ref{thm:hvec}, for $i\in [d]$, $h_i(\mathcal{C})=h_i(\link(\mathbf{v_P},\mathcal{C}_P))$. In addition, since the $h$-vector for $\link(\mathbf{v_P},\mathcal{C}_P)$ does not have an $h_{d+1}$ component, $h_{d+1}(\mathcal{C}_P)=0$.
\end{proof}

Notice that the Euler characteristic of an $l$-dimensional polyhedral complex $\mathcal{C}$ can be found from the $f$-vector by $\chi (\mathcal{C} ) = \sum _{i=0} ^l (-1)^j f_j(\mathcal{C} )$. We can use this formula to show that, for a pointed triangulation $\mathcal{C}_P$ of a $d$-dimensional polytope $P$, $h_d (\mathcal{C}_P)=0$. Recall that we have already seen in Lemma \ref{lem:eq} that $h_{d+1} (\mathcal{C}_P)=0$.

\begin{lemma}\label{lem:hd}
For any $d$-dimensional polytope $P$ with a pointed triangulation $\mathcal{C}_P$, $h_d (\mathcal{C}_P)=0$.
\end{lemma}
\begin{proof}
Assume that $\mathcal{C}_P$ is the $V(P)$-pointed triangulation of $P$. By Theorem \ref{lem:eq} it is enough to show that $h_d (\link(\mathbf{v_P},\mathcal{C}_P) )=0$. By the definition of the $h$-vector $$h_d= \sum _{i=0}^d (-1)^{k-i} \binom{d-i}{0} f_{i-1}=\sum _{i=0}^d (-1)^{k-i} f_{i-1},$$ where $h_d=h_d(\link(\mathbf{v_P},\mathcal{C}_P) )$ and $f_{i-1}(\link(\mathbf{v_P},\mathcal{C}_P) )$. Also since $\link(\mathbf{v_P},\mathcal{C}_P)$ is contractable and has Euler characteristic 1, $1=\sum _{i=0}^{d-1} (-1)^i  f_i$. Therefore, 
$$ h_d=\sum _{i=0}^d (-1)^{d-i} f_{i-1}=(-1)^d \left[ f_{-1} - \sum_{i=0}^{d-1} (-1)^i f_i \right]=(-1)^d ( 1 - 1 )=0.$$
\end{proof}

\section{Polytope Numbers}
Below is a formal definition of the polytope number.

\begin{definition}
\label{def:poly}
For a $d$-polytope $P$, we define a sequence of polytope numbers $P(n)$ and interior polytope numbers $P(n)^\#$ by double induction on $d$ and $n$. First assume the set of apexes $V(P)$ has been choosen. When $d=0$, we define $P(0)=P(0)^\#=0$ and $P(n)=P(n)^\#=1$ for $n\geq 1$.

Now we use the induction with $d>0$ and assume that for each $F\in \mathcal{F}_k (P)$ with $k<d$, the sequences $F(n)$ and $F(n)^\#$ are defined. We start the sequences with $P(0)=P(0)^\#=0$, $P(1)=1$, and $P(1)^\#=0$. Using induction on $n$, we complete the sequences with the following formula for $n\geq 2$:
$$P(n)=P(n-1)+\sum _{\mathbf{v}_P\notin F \in \mathcal{F}(P)} F(n)^\# \textrm{ and }P(n)^\#=P(n)-\sum _{F\in \mathcal{F}(P)\setminus \{P\} } F(n)^\# .$$
\end{definition}
Notice that by rearranging the last summation we get the following formula for $n\geq 2$: $$P(n)=\sum_{F\in \mathcal{F}(P) } F(n)^\#.$$
If the polytope is not vertex transitive, the polytope number sequences will depend on the choice of vertices $V(P)$ and so should be called the $V(P)$-polytope numbers.

We need to look especially at the simplex number sequences. Just as we used $\alpha $ to stand for a simplex, we will use it to represent the simplex number sequences. So $\alpha ^d (n)$ will represent the number sequence for the $d$-dimensional simplex.

Recall from H. K. Kim's work \cite{hkk:pn} the following formulas:
$$\alpha^d (n) = \sum _{i=1} ^n \alpha ^{d-1}(i)=\binom{d+n-1}{d}$$
$$\alpha^d (n)-\alpha^d (n-1)=\alpha^{d-1} (n).$$
The second formula can be rewritten $\alpha^d(n) - \alpha^{d-1} (n) = \alpha ^d(n-1)$. Also note that the facets of a $d$-dimensional simplex are $(d-1)$-dimensional simplices and so this formula can be thought of as cutting a facet off of a simplex. By repeating this process, we get the following result formula, which we will refer to as the \emph{facet-cut} formula:
$$\alpha ^d (n) - \sum_{i=0}^{k-1} \alpha ^{d-1} (n-i)= \alpha^d (n-k).$$
The facet-cut formula holds for any $k\leq n$, but it has geometric meaning only when $k\leq d+1$ since a simplex of dimension $d$ has $d+1$ facets to cut. In addition if all of the facets are cut, then the result is the interior of the simplex, which means that $\alpha^d (n) ^\# = \alpha ^d (n-d-1).$

\begin{lemma}
\label{lem:inttri}
Let $P$ be a polytope with a $V(P)$-pointed triangulation $\mathcal{C}_P$. We will let $\mathcal{C}_F$ be the corresponding pointed triangulation of a face $F\in \mathcal{F}(P)$. Using $\alpha ^i _\beta $ as an $i$-dimensional simplex in the triangulation, we have
$$ P(n)=\sum_{\alpha^i _\beta \in \mathcal{C}_P } \alpha ^i (n)^\#.$$
\end{lemma}
\begin{proof}
We will use induction on the dimension of $P$. Clearly this is true for a polytope of dimension 1 since every polytope of dimension 1 is a simplex. Now assume that the statement is true for every polytope of dimension less than $d$ where $P$ is a $d$-dimensional polytope. In particular the statement holds for every face of $F\in \mathcal{F}(P)\setminus \{P\}$. Let $F\in \mathcal{F}(P)\setminus \{P\}$, then 
$$ F(n)=\sum_{\alpha^i _\beta \in \mathcal{C}_F } \alpha ^i (n)^\#.$$
Also by induction,
$$ F(n)^\#=\sum_{\alpha^i _\beta \in \mathcal{C}_F\setminus \mathcal{C}_{\partial F}} \alpha ^i (n)^\#.$$
Since $P(n)=P(n-1)+\sum _{\mathbf{v}_P \notin F \in \mathcal{F}(P) } F(n)^\#$, we see that
$$P(n)=1+ \sum _{j=2}^n \sum _{\mathbf{v}_P \notin F \in \mathcal{F}(P) } F(j)^\#.$$
But this shows that 
$$P(n)=1+ \sum _{j=2}^n \sum _{\mathbf{v}_P \notin F \in \mathcal{F}(P) } \sum_{\alpha^i _\beta \in \mathcal{C}_F\setminus \mathcal{C}_{\partial F}} \alpha ^i (j)^\#.$$
Since every simplex is interior to exactly one face of $P$, we can write the expression as follows: 
$$P(n)=1+ \sum _{j=2}^n \sum _{\mathbf{v}_P \notin \alpha^i_\beta  \in \mathcal{C}_P } \alpha^i (j)^\#.$$
Notice that as in the proof of Theorem \ref{thm:ptri}, each $\alpha ^i _\beta \in \mathcal{C}_P$ with $\mathbf{v}_P \notin \alpha ^i _\beta$ corresponds uniquely to an $\alpha ^{i+1} _\beta =\conv (\{\mathbf{v}_P\} \cup \alpha^i_\beta )$ where $\mathbf{v}_P \in \alpha ^{i+1} _\beta\in \mathcal{C}_P$. Also notice that $\alpha ^{i+1} (n)^\# = \sum _{j=2}^{n-1} \alpha ^i (j) ^\# $ for $i\geq 1$ (since $\alpha ^i (j)^\#=0$ for $j<2$). Putting these sums together, we see that
\begin{align*}
P(n)&=1+ \sum _{j=2}^n \sum _{\mathbf{v}_P \notin \alpha^i_\beta  \in \mathcal{C}_P } \alpha^i (j)^\#\\
&=1+ \sum _{\mathbf{v}_P \notin \alpha^i_\beta  \in \mathcal{C}_P } \sum _{j=2}^n \alpha^i (j)^\#\\
&=1+ \sum _{\mathbf{v}_P \notin \alpha^i_\beta  \in \mathcal{C}_P }  \alpha^i (n)^\# +\sum _{\mathbf{v}_P \in \alpha^{i+1}_\beta  \in \mathcal{C}_P }  \alpha^{i+1} (n)^\#\\
&= \sum _{\alpha^i _\beta \in \mathcal{C}_P } \alpha_i (n)^\#
\end{align*} 
where in the last summation the 1 becomes the $\alpha^0 (n)^\#$ arising from the vertex $\mathbf{v}_P$.
\end{proof}

\begin{theorem}
\label{thm:numseq}
Let $P$ be a $d$-polytope with a $V(P)$-pointed triangulation $\mathcal{C}_P$. In addition let $h_j$ for $j\in [d]$ be the $h$-vector components for this polytope. Then 
$$P(n)=\sum _{j=0}^{d-1} h_j  \alpha^d (n-j).$$
\end{theorem}
\begin{proof}
By Lemma \ref{lem:inttri} we have $P(n)=\sum _{\alpha^i _\beta \in \mathcal{C}_P } \alpha_i (n)^\#$. Recall that the $f$-vector component, $f_i$, gives the number of simplices of dimension $i$ in the simplicial complex $\mathcal{C}_P$. Thus $P(n)=\sum _{i=0}^d f_i \alpha^i (n)^\#$.

Recall $f_i= \sum _{j=0} ^{i+1} h_j \binom{d+1-j}{i+1-j}$. See \cite[Section 8.3]{gz:lop}. 

\begin{align*}
P(n)&=\sum _{i=0}^d f_i \alpha^i (n)^\#\\
&=\sum _{i=0}^d \sum _{j=0} ^{i+1} h_j \binom{d+1-j}{i+1-j} \alpha^i (n)^\#\\
&=\sum _{j=0}^{d+1}  h_j \left[ \sum _{i=j-1} ^{d} \binom{d+1-j}{i+1-j} \alpha^i (n)^\# \right].
\end{align*} 
Recall from Lemma \ref{lem:hd} that $h_d=0=h_{d+1}$. So to finish the proof, we need to show that $$\sum _{i=j-1} ^{d} \binom{d+1-j}{i+1-j} \alpha^i (n)^\# = \alpha^d (n-j).$$ This result follows from an identity called Generalized Kim's Identity \cite[Lemma 2]{j:rspn}. For completeness we will show the result here.
\begin{align*}
\sum _{i=j-1} ^{d} \binom{d+1-j}{i+1-j} \alpha^i (n)^\# &=\sum _{i=0} ^{d-j+1} \binom{d+1-j}{i} \alpha^{j-1+i} (n)^\#\\
&= \sum _{i=0} ^{d-j+1} \binom{d+1-j}{i} \binom{n-2}{i+j-1}\\
&= \sum _{i=0} ^{d-j+1} \binom{d+1-j}{i} \binom{n-2}{n-1-i-j}\\
&= \binom{d-j+n-1}{n-1-j} \textrm{  By Vandermonde's convolution}\\
&= \binom{d-j+n-1}{d}=\alpha ^d (n-j).
\end{align*}
\end{proof}
\section{Interior Polytope Numbers}

Let $P$ be a polytope with a $V(P)$ pointed triangulation $\mathcal{C}_P$. The complex $\mathcal{I}_P$ is the collection of simplices $F\in \mathcal{C}_P$ such that $F\notin G$ for all $G\in \mathcal{F}(P)\setminus \{P\}$. Notice that this complex is not a polytopal complex and that $\mathcal{C}_P=\mathcal{I}_P \uplus \mathcal{C}_{\partial P}$.

We will spend this section developing a theory for the complex $\mathcal{I}_P$ similar to the theory used in the two previous sections for the complex $\mathcal{C}_P$. We will then use this theory to prove Theorem \ref{Thm:Main}. Throughout the section we will assume that $P$ is a convex polytope with a $V(P)$-pointed triangulation $\mathcal{C}_P$. We will also define $\mathcal{I}_P$ as above. 

Let $\mathbf{x}$ be in general position with respect to $\mathcal{C}_P$ as in the proof of Theorem \ref{thm:part}. For each $F \in \mathcal{F}_d (\mathcal{C}_P)$, Let $R_F$ be the set of facets of $F$ visible from $\mathbf{x}$. Let $Q_F= \mathcal{F}_{d-1} (F)\setminus R_F$. Thus $Q_F$ is the set of facets of $F$ that are not visible from $\mathbf{x}$. Let $D_F$ be the set of vertices of $F$ that are opposite a facet in $Q_F$.

\begin{lemma}
\label{lem:intpart}
Let $P$ be a convex $d$-polytope with a $V(P)$-pointed triangulation $\mathcal{C}_P$. 
Define $\mathcal{I}_P$ as above and for each $F\in \mathcal{F}_d (\mathcal{C}_P)$ define $D_F$ as above. Then  
$$\biguplus _{F\in \mathcal{F}_d(\mathcal{C}_P)} [D_F, F] = \mathcal{I}_P.$$
\end{lemma}

\begin{proof}
Let $H\in \mathcal{C}_P\setminus \mathcal{C}_{\partial P }$. There is a unique $d$-simplex $F_0$ such that the ray from $\mathbf{x}$ through $\mathbf{y}$ intersects $F_0$ at points after $\mathbf{y}$ for all $\mathbf{y}$ in the interior of $H$. Each facet of $F_0$ either contains $H$ or is opposite a vertex contained in $H$. The facets of $F_0$ that contain $H$ are exactly those that are visible from $\mathbf{x}$. Thus $H\in [D_{F_0}, F_0]$.

Suppose that $H\in [D_F,F]$ for some $F\in \mathcal{F}_d (\mathcal{C}_P)$. Then $D_F\subseteq H$, which implies that $H$ is not contained in any facet that is not visible from $\mathbf{x}$. Thus if we take an interior point $\mathbf{y}$ of $H$ and move into any facet of $F$ containing $H$, the ray from $x$ will encounter $F$ after reaching that facet. This implies that $F$ is the $F_0$ from the previous paragraph. In this way we see that if $F,F' \in \mathcal{F}_d(\mathcal{C}_P)$ are not the same, $[D_F,F]\cap [D_{F'},F']=\emptyset$. Also each facet in $\mathcal{F}_{d-1}(\mathcal{C}_{\partial P})$ is only contained in one $d$-simplex of $\mathcal{C}_P$ and is not visible from $\mathbf{x}$ as a facet of that simplex. Thus for each $F\in \mathcal{F}_d (\mathcal{C}_P)$, $[D_F,F]\cap \mathcal{C}_{\partial P}=\emptyset $. 
\end{proof}

\begin{definition}
Let $P$ be a convex $d$-polytope with a $V(P)$-pointed triangulation $\mathcal{C}_P$. 
Define $\mathcal{I}_P$ as above and for each $F\in \mathcal{F}_d (\mathcal{C}_P)$ define $D_F$ as above. The $k$-vector of $\mathcal{I}_P$ is $\mathbf{k} (\mathcal{I}_P)=(k_0, k_1, \dots , k_{d+1} )$ where $k_i=\left| \left\{ F\in \mathcal{F}_d(\mathcal{C}_P) : |D_F|=i ) \right\}\right|$. The $e$-vector of $\mathcal{I}_P$ is $\mathbf{e}(\mathcal{I}_P)=(e_0, e_1, \dots , e_{d})$ where $e_i=\left| \left\{ F\in \mathcal{I}_P : \dim(F)=i \right\} \right|$.
\end{definition}

\begin{corollary}
\label{cor:hk}
Let $P$ be a convex $d$-polytope with a $V(P)$-pointed triangulation $\mathcal{C}_P$. Define $\mathcal{I}_P$ as above with the $k$ and $h$ vectors defined for $\mathcal{I}_P$ and $\mathcal{C}_P$, respectively.
Then $k_{i}=h_{d+1-i}$ for $i\in [d+1]$.
\end{corollary}

\begin{proof}
Using the definitions above, $|D_F|+|G_F|=d+1$ for each $F\in \mathcal{F}_d(\mathcal{C}_P)$. So by Theorem \ref{thm:hvec}, $k_{d+1-i}=h_{i}$ for all $i\in [d+1]$. Equvalently $k_{i}=h_{d+1-i}$ for $i\in [d+1]$.
\end{proof}

Recall
$$ P(n)^\#=\sum_{\alpha^i _\beta \in \mathcal{I}_P} \alpha ^i (n)^\#$$
and $$ P(n)^\#=P(n) - \sum _{\alpha^i _\beta  \in \mathcal{C}_{\partial P}} \alpha ^i (n)^\#.$$

\begin{theorem}
\label{thm:int}
Let $P$ be a polytope with a $V(P)$-pointed triangulation $\mathcal{C}_P$. Let $k_j$ be the components of the $k$-vector for this triangulation, then 
$$P(n)^\# = \sum _{j=0}^{d+1} k_j \alpha^d (n-j).$$
\end{theorem}
\begin{proof}
Let $F\in \mathcal{F} (\mathcal{C}_P )$. If $|D_F|= j$, then the number of $i$-dimensional faces in the interval $[D_F, F]$ will be $\binom{d+1-j}{i+1-j}$. 
Thus $e_i = \sum _{j=0}^{i+1} k_j \binom{d+1-j}{i+1-j}$. So we see that 
\begin{align*}
P(n)^\# &=\sum_{\alpha^i _\beta \in \mathcal{I}_P} \alpha ^i (n)^\# \\
&= \sum_{i=0}^{d} e_i \alpha ^i (n)^\#\\
&= \sum_{i=0}^{d} \sum _{j=0}^{i+1} k_j \binom{d+1-j}{i+1-j} \alpha^i (n) ^\# \\
&= \sum _{j=0} ^{d+1} k_j \left[ \sum_{i=j-1}^d \binom{d+1-j}{i+1-j} \alpha^i (n) ^\# \right] \\
&= \sum _{j=0} ^{d+1} k_j \alpha ^d (n-j)
\end{align*}
where the last equality is shown in the proof of Theorem \ref{thm:numseq}.
\end{proof}

\begin{proof}[Proof of Theorem \ref{Thm:Main}]
By Theorem \ref{thm:int}
$$P(n)^\# = \sum _{j=0}^{d+1} k_j \alpha^d (n-j).$$ In addition by Corollary \ref{cor:hk}, $k_j= h_{d+1-j}$. So we see that  
$$P(n)^\# = \sum _{j=0}^{d+1} h_{d+1-j} \alpha^d (n-j)=\sum _{j=0}^{d-1} h_{j} \alpha^d (n-(d+1-j)).$$
We get the last equality by re-indexing and recalling that $h_d=h_{d+1}=0$.
\end{proof}



\end{document}